\documentclass[11pt]{amsart}

\usepackage{
amsfonts,
amsmath,
latexsym,
amssymb,
enumerate,
verbatim,
centernot,
accents,
color,
epigraph,
tikz,
simplewick,}

\usepackage{tikz} \usetikzlibrary{arrows}\usepackage{graphicx}

\newcommand{\labbel}[1]{\label{#1} [[{\bf #1}]]}  
\renewcommand{\labbel}{\label}

\newcommand{\mathbff}{}
\newcommand{\ranglee}{\big \rangle }
\newcommand{\langlee}{\big \langle }

\newtheorem{theorem}{Theorem}

\newtheorem{proposition}[theorem]{Proposition} 
 
\newtheorem{corollary}[theorem]{Corollary}

\newtheorem*{claim*}{Claim}

\newtheorem*{theorem*}{Theorem}
\newtheorem*{proposition*}{Proposition}
\newtheorem*{corollary*}{Corollary}
\newtheorem*{lemma*}{Lemma}
\newtheorem*{scholion*}{Scholion}

\theoremstyle{definition}

\theoremstyle{remark}
\newtheorem{remark}[theorem]{Remark}

\newtheorem*{remark*}{Remark}
\newtheorem*{remarks*}{Remarks}

\newtheorem*{observation*}{Observation}

 \allowdisplaybreaks[1]

\numberwithin{equation}{section}

\begin{document}

\title{Non-generators in complete lattices and semilattices}

\author{Paolo Lipparini} 

\email{lipparin@axp.mat.uniroma2.it}

\urladdr{http://www.mat.uniroma2.it/~lipparin}

\address{Dipartimento di Matematica\\Viale della  Ricerca
Genereticolata\\Universit\`a di Roma ``Tor Vergata'' 
\\I-00133 ROME ITALY\\
ORCiD: 0000-0003-3747-6611}



\date{\today}

\thanks{
Work performed under the auspices of G.N.S.A.G.A. Work 
partially supported by PRIN 2012 ``Logica, Modelli e Insiemi''.
The author acknowledges the MIUR Department Project awarded to the
Department of Mathematics, University of Rome Tor Vergata, CUP
E83C18000100006.}

\begin{abstract}
As well-known, in a finitary algebraic structure 
 the set $\Gamma$ of all the non-generators is the intersection
 of all the maximal proper
substructures. In particular, $\Gamma$ 
is a substructure.

We show that the corresponding statements hold
for complete semilattices but fail for complete lattices,
when as  the notion of   substructure we take complete
subsemilattices and  complete sublattices, respectively.
\end{abstract}

\maketitle  
    
We shall consider complete 
lattices and complete semilattices,
whose substructures are taken to be complete sublattices and complete 
subsemilattices, respectively.
In a \emph{complete sublattice}
both meets and joins, possibly infinitary, 
should agree with meets and joins of the original structure.
In a \emph{complete subsemilattice}
only meets are required to 
 agree with  the original meets.
We refer to \cite{G} for basics about lattices and semilattices. 

To simplify  notation, we shall not distinguish 
between a structure and its underlying set.
This little abuse of notation will never produce ambiguity.
If, say, $\mathbff L$ is a complete lattice
and $X \subseteq L$, then 
$\langlee X \ranglee$   
denotes the complete sublattice \emph{generated by} 
$X$, that is, the intersection
of all the complete sublattices of $\mathbff L$  
containing $X$.
Of course, when dealing with complete semilattices,
$\langlee X \ranglee$ generally turns out to be a different
set; it will always be clear from the context whether we are 
working in the setting of lattices or of semilattices.
The following definitions apply to both settings.

An element $a \in L$ is a \emph{non-generator} if,
 for every $X \subseteq L$, it happens that 
 $\langlee X, a\ranglee = L$ implies
 $\langlee X \ranglee = L$.
Otherwise, $a$ is called a \emph{relative generator}.
Thus $a$ is a relative generator
if there is some  $X \subseteq L$ such that 
 $\langlee X, a\ranglee = L$ but not
 $\langlee X \ranglee = L$.
Here, as usual, 
 $\langlee X, a\ranglee $ is an abbreviation for
 $\langlee X \cup \{ a \}  \ranglee $.
An element
$a$ is \emph{indispensable} if $a$ belongs to every generating  set,
that is, $\langlee X \ranglee = L$ implies $a \in X$. 
In particular, every indispensable element
is a relative generator (take $X=L \setminus \{ a \} $). 
Notice that some element $a$ is indispensable
if and only if  $L \setminus \{ a \}  $
is a substructure.  

Of course the above notions can be considered---and originally
 have been considered---in various distinct algebraic settings.
See \cite{BS,Jz,KV} for further details and references.
It is immediate to see that a non-generator belongs
to the intersection of all the proper maximal 
substructures, if any.
The converse holds for algebraic structures 
whose operations are finitary, but not necessarily in the infinitary case
\cite{H}. We show that a counterexample can be realized already 
in the classical and well-studied setting of complete lattices.
On the other hand,  complete semilattices share the good behavior
of finitary algebraic structures, as far as non-generators
are considered. Let us mention that we deal with completeness
just for notational simplicity; in both cases asking just for the
existence of countable meets or joins is enough.
A few further variations shall be presented in Remark \ref{card} below.  

The arguments presented in this note are substantially different from
the proof in \cite{H}, which is indirect and relies on \cite{Gr}.
In order to make a comparison, we briefly sketch the arguments from
\cite{H}; this shall not be needed in the sequel. 

If $A$ is an algebraic structure, let $S(A)$
be the closure system associated to  the set of substructures of $A$,
 $\Gamma(A)$ be the set of the non-generators
of $A$ and $\Phi(A)$ be the intersection of 
all the proper maximal substructures of $A$,
setting $\Phi(A)=A$ if $A$ has no proper maximal substructure.
 Thus 
 $\Gamma(A) = \Phi(A)$ in any finitary structure,
$\Phi(A) $ is a substructure of $ A$
and $\Gamma(A) \subseteq  \Phi(A)$ always.

In \cite{H} a closure system $S$ is constructed such that
if $A$ is a structure and $S=S(A)$, then $\Gamma(A)$  is not a substructure
of $A$, and moreover 
$ \langle \Gamma(A) \rangle $ is strictly
contained in $\Phi(A)$.
By \cite[Theorem 1]{Gr}, for every closure system $S$, 
there is indeed some infinitary algebraic structure 
$A$ such that $S=S(A)$. It follows that
$\Gamma(A) \subsetneq \Phi(A)$ 
might actually happen
for infinitary algebraic structures. 
To the best of our knowledge,
no other example of $\Gamma(A) \subsetneq \Phi(A)$
has appeared before in the literature.
Further examples will appear in \cite{Lp}.

The structures considered
in \cite{Gr} have infinitely many operations
and each operation is  everywhere a projection, except
when applied to a single specific sequence. 
On the other hand, here we can
equivalently  work with
countably complete lattices, hence our counterexamples 
can be taken to be structures
with just two operations. Actually, as we shall show
in Corollary \ref{unasola}, our examples can be easily modified
in order to work with just one operation
depending on countably many arguments.

Conventionally, the intersection of an empty
family of substructures  of some structure $\mathbff L$
is taken to be $\mathbff L$  itself. 
In the standard definition of a complete semilattice
(lattice) the empty subset is required to have a meet (and a join);
in other words, complete semilattices (lattices) are required
to have a maximum (and a minimum). 
 We adopt the above convention; in particular, 
a complete subsemilattice of some complete semilattice
shares the maximum with the parent structure,
and similarly for lattices. In any case, 
in all the examples below, maxima and minima are non-generators
even under the alternative convention under which
the meet or join of the empty subset need not exist
or need not be preserved.
It follows that our results hold irrespective
of the convention about the meet and join of the empty set.
See Remark \ref{pf}(b).

\begin{proposition} \labbel{seml}
In every complete
semilattice the set $\Gamma$  of non-generators is a 
complete subsemilattice and $\Gamma$  is the 
intersection of all the  maximal proper 
complete subsemilattices.
 \end{proposition} 

 \begin{proof} 
We first prove that an element $a$ of  some complete
semilattice $\mathbff S$ is  a non-generator
if and only if $a$  is \emph{meet reducible}
(here meets are always allowed to be infinitary), 
that is, there is a subset $Y \subseteq S$
such that $a \notin Y$   and $a = \bigwedge Y$.

Indeed, suppose that $a$ is meet reducible, as witnessed 
by $Y$.  If $X \subseteq S$ and $\langlee X, a \ranglee = S$,
then  $y \in \langlee X, a \ranglee $, for every $y \in Y$.
But then  $y \in \langlee X \ranglee $, since $y>a$, hence 
$a$ cannot contribute to the generation of $y$. Then 
also  $a \in \langlee X \ranglee $, since 
 $a = \bigwedge Y$ and $ \langlee X \ranglee $
is a complete subsemilattice.
Hence  $  \langlee X\ranglee= \langlee X, a \ranglee = S $.
Thus $a$ is a non-generator. 

On the other hand, if
$a$ is meet irreducible,
then $S \setminus \{ a \} $ is a complete
subsemilattice 
of $\mathbff S$, hence $a$ fails to be a non-generator,
actually, $a$ is indispensable.

We have proved a bit more:
an element of some complete semilattice is either
indispensable, or a non-generator (the finitary case appears in \cite{K}).
Thus the set $\Gamma$ of the non-generators is the intersection
of the complete subsemilattices of the form
$S \setminus \{ b \} $, with $b$ indispensable.  
No maximal proper complete subsemilattice of a different kind
exists, since all the non-generators belong to every
maximal proper complete subsemilattice and if some element $b$
fails to be a non-generator, then
$S \setminus \{ b \} $ is a complete
subsemilattice 
of $\mathbff S$.
Hence $\Gamma$ is the intersection
of the maximal proper complete subsemilattices of $\mathbff L$.
In particular, $\Gamma$  is a complete subsemilattice.
\end{proof} 

See Remark \ref{card} below for some variations
on Proposition \ref{seml}.  

Besides semilattices, there are many situations 
in which $\Gamma(A) = \Phi(A)$
 even for infinitary algebras. Indeed, the proof of Proposition \ref{seml}
shows that this is the case for every structure in which 
every element is either indispensable or a non-generator.
As another example, $\Gamma(A) =  \Phi(A)$
holds for finite (= having finite domain) structures, 
since in this case any substructure can be extended
to a maximal one, and then the classical argument applies.
We are not aware of any systematic study
of infinitary structures for which $\Gamma(A) = \Phi(A)$ holds.
In \cite{Lp} we show that every structure with
at least one infinitary operation  can be embedded into some structure $B$ 
such that $\Gamma(B)$  is not a substructure of $B$,
but can also  be embedded into some structure $C$ 
such that $\Gamma(C) = \Phi(C)$.
Moreover, in \cite{Lp} we present non-trivial examples of structures $D$ 
such that $\Gamma(D^+)$ is a substructure of $D^+$,
for every expansion $D^+$ of $D$.

We now show that lattices behave in a way different from
semilattices. The next proposition is rather simple and shows that 
in a distributive complete lattice the intersection of all the maximal 
proper complete sublattices 
might be strictly larger than the set of  the non-generators.

A complete linearly ordered set
is considered as a complete lattice
endowed with the operations of 
  $\sup$ and  $\inf$. 
 In particular, this applies to 
the set $\{ 0, 1\}$ with the standard
order. 
The reader familiar with ordinals
will recognize that the linearly ordered sets 
$\mathbff K$ in, respectively, Proposition \ref{latprop}
and Theorem \ref{lat} below  
are isomorphic to the ordinals $ \omega+1$
and $ \omega^2+1$.    
However, we shall  need no aspect of the
theory of ordinals, hence we give  explicit definitions
from scratch.

\begin{proposition} \labbel{latprop} 
 Let $\mathbff K$  be the complete lattice 
obtained from the linearly ordered set $\mathbb N$ 
by adding a top element 
$ \omega$. Let $\mathbff L$ be the lattice product
$\mathbff K \times \{ 0,1 \} $. 

In the complete lattice $\mathbff L$ 
the element $( \omega , 0) $ is a relative generator,
but $( \omega , 0) $ belongs to all the
 maximal proper complete sublattices of $\mathbff L$.   
The set $\Gamma$ of all the non-generators is a 
complete sublattice of $\mathbff L$.  
\end{proposition}  

\begin{center}
\begin{tikzpicture}[-,>=stealth',auto,node distance=3cm,
thick,main node/.style={circle,scale=0.3,draw},
mnode/.style={scale=0.3}]
\node[main node,label={south:$(0,0)$}] (00){};
\node[main node,label={east:$(1,0)$}] (10) [above right of=00]{};
\node[main node,label={east:$(2,0)$}] (20) [above right of=10]{};
\node[main node,label={east:$(3,0)$}] (30) [above right of=20]{};
\node[mnode] (u0) [above right of=30]{};
\node[main node,label={east:$( \omega ,0)$}] (w0) [above right of=u0]{};
\draw (00) -- (10)-- (20)-- (30);
\draw (30) [dotted] -- (u0);
\node[main node,label={west:$(0,1)$}] (01) [above left of=00]{};
\draw (00)--(01);
\node[main node,label={west:$(1,1)$}] (11) [above left of=10]{};
\draw (10)--(11);
\node[main node,label={west:$(2,1)$}] (21) [above left of=20]{};
\draw (20)--(21);
\node[main node,label={west:$(3,1)$}] (31) [above left of=30]{};
\draw (30)--(31);
\node[mnode] (u1) [above right of=31]{};
\node[main node,label={north:$( \omega ,1)$}] (w1) [above right of=u1]{};
\draw (w0)--(w1);
\draw (01) -- (11)-- (21)-- (31);
\draw (31) [dotted] -- (u1);
\end{tikzpicture}
\smallskip

The lattice $L$ in Proposition \ref{latprop}.
\end{center}

\begin{proof} 
Observe that
$M= (\mathbb N \times \{ 1 \}) \cup \{ (0,0) \}   $
is  a complete proper sublattice of $\mathbff L$.
Here  we include the bottom element $(0,0)$
according to the convention that complete sublattices share
the minimum with their parent lattice.
Since $\langlee M, ( \omega , 0) \ranglee =L$,
then $( \omega , 0)$ is a relative generator. 

Now we check that $( \omega , 0) $ belongs to all the
 maximal proper complete sublattices of $\mathbff L$.
Suppose by contradiction that $P$ is a maximal complete sublattice
and $( \omega , 0) \notin P$. Then there is some $n \in \mathbb N$ 
such that $( m, 0) \notin P$, for every $m \geq n$, since 
$P$ is assumed to be complete. Now notice that  
$Q = L \setminus \{ \, (m,0) \mid  m \in K, m > n\, \} $ 
is a proper complete sublattice of $\mathbff L$.
Moreover, $P \subsetneq Q$, since
$(n,0) \in Q \setminus P$. Hence $P$ is not maximal and
we have obtained the desired contradiction. 
Notice that if we consider $\mathbff L$ as a finitary lattice, 
instead, then the subset
$L \setminus \{ ( \omega ,0) \} $ is a sublattice. However this subset
is not a complete sublattice,
since it is not closed with respect to the operation of taking infinitary joins.

Next, observe that  $L \setminus \{ (0,1)\} $, as well as  the subsets of the form 
$L \setminus \{ (n,0), (n,1) \} $, for $n \in \mathbb N$, $n \geq 1$,
are maximal proper complete sublattices.
Since any non-generator is contained in all the maximal proper
complete sublattices, then $\Gamma \subseteq \{  (0,0), ( \omega , 1)  \} $,
since we have already proved that $( \omega , 0) \notin \Gamma $. 
But $ (0,0) $ and $  ( \omega , 1)$ are obviously
non-generators, thus $\Gamma= \{  (0,0), ( \omega , 1)  \} $,
which  is a complete sublattice.
\end{proof} 

\begin{remark} \labbel{pf} 
 Some remarks about the proof of 
Proposition \ref{latprop} 
are in order.

(a)
First, the  proof shows that the intersection  $\Phi$    
of all the maximal proper
complete sublattices is $ \{  (0,0), ( \omega , 0), ( \omega , 1)  \}$,
since every non-generator belongs to $ \Phi$,
we have showed directly that $( \omega , 0) \in \Phi$
and, for every remaining element $\ell$  of $L$, 
we have exhibited a maximal proper
complete sublattice $P$ such that $\ell \notin P$. 
In particular, $\Phi \neq \Gamma $.

(b)
In the present lattice $L$ the bottom element $ (0,0) $ and
the top element $  ( \omega , 1)$ are 
non-generators even under the alternative convention
that sublattices need not respect minima and maxima.
Indeed, since, say, $ ( \omega , 1)$
is the top element of $L$, then
 $\langle X,  ( \omega , 1) \rangle = \langle X  \rangle \cup \{ ( \omega , 1)  \} $,
for every $X \subseteq L$.
Thus if   $\langle X,  ( \omega , 1) \rangle = L$,
 then $  (0,1),  ( \omega , 0) \in \langle X\rangle$,
hence  $\langle X \rangle = \langle X,  ( \omega , 1) \rangle $,
since $ (0,1) \vee  ( \omega , 0)  =( \omega , 1) $. 
This shows that $ ( \omega , 1)$ is a non generator in $L$
and a similar argument applies to $(0,0)$.
(On the other hand, for example, in a complete linearly ordered lattice
the bottom and the top elements are non-generators
if and only if we assume the convention that 
sublattices   respect minima and maxima.)

(c)
Finally, we show that all the  maximal proper complete sublattices
of $L$ are the ones described in the  third paragraph 
of the proof.
We show that if $P$ is a
maximal proper complete sublattice of $L$,
$n \in \mathbb N$ and $n \geq 1$, then  
$(n,0) \in P$ if and only if $(n,1) \in P$.

Suppose by contradiction that, say, 
$(n,0) \in P$ and $(n,1) \notin P$.
Then $(0,1), \dots, (n-1,1) \notin P$, since  
$(n,0) \in P$ and $(i,1) \vee (n,0)=(n,1)$, for $i <n$.
Hence if $Q= L \setminus \{ (0,1), \dots, (n,1) \} $,
then $P \subseteq Q$, but this contradicts the maximality
of $P$, since $Q$ is a complete sublattice 
which is not maximal. Indeed, since $n \geq 1$,
then $ L \setminus \{ (0,1), \dots, (n-1,1) \} $ is a proper complete 
sublattice of $L$ extending $Q$.
Symmetrically,
if $n \in \mathbb N$, $(n,1) \in P$ and  $(n,0) \notin P$,
then $P$ is contained in the sublattice
$L \setminus \{ \, (i,0) \mid n \leq i \leq \omega   \, \} $,
which is not maximal. 

Since we have showed that 
$(0,0) $, $ ( \omega , 0) $ and $  ( \omega , 1)$ 
belong to every maximal complete sublattice, then  the 
maximal complete sublattices of $L$ are exactly 
  $L \setminus \{ (0,1)\} $ and  
$L \setminus \{ (n,0), (n,1) \} $, for  $n \geq 1$.
\end{remark}

The next result needs a bit more work and shows that 
in a distributive complete lattice the set of the non-generators 
might even fail to be a complete sublattice.
We shall denote by 
$\mathbff L \ltimes \mathbff M$
the lexicographic product of two linearly ordered sets. 
Elements of $\mathbff L \ltimes \mathbff M$ shall be denoted by 
$ a \ltimes b$, for $a \in L$ and $b \in M$.
In the next theorem we shall consider an extension of
$(\mathbb N \ltimes\mathbb N) \times \{ 0,1 \}$,
where $\times $ is  the standard product of lattices.
Elements of   $(\mathbb N \ltimes\mathbb N) \times \{ 0,1 \}$
shall be denoted by
$(m \ltimes n, 0)$ and
$(m \ltimes n, 1)$.

\begin{theorem} \labbel{lat}
 Let $\mathbff K$  be the linearly ordered set
obtained from the lexicographic 
product $\mathbb N \ltimes \mathbb N$ by adding a top element 
$ \omega^2$. Let $\mathbff L$ be the lattice product
$\mathbff K \times \{ 0,1 \} $. 

In the complete lattice $\mathbff L$
the element $( \omega^2 , 0) $ is a relative generator,
but $( \omega^2 , 0) $ belongs to 
the complete sublattice  generated 
by the set  $\Gamma$ of the non-generators of $\mathbff L$.
In particular, $\Gamma$ is not a complete sublattice of $\mathbff L$.  
 \end{theorem} 

\begin{center}
\begin{tikzpicture}[-,>=stealth',auto,node distance=3cm,
thick,main node/.style={circle,scale=0.2,draw},
mnode/.style={scale=0.2}]
\node[main node,label={south:$(0 \ltimes 0,0)$}] (00){};
\node[main node,label={east:$(0 \ltimes 1,0)$}] (10) [above right of=00]{};
\node[main node,label={east:$(0 \ltimes 2,0)$}] (20) [above right of=10]{};
\node[main node,label={east:$(0 \ltimes 3,0)$}] (30) [above right of=20]{};
\node[mnode] (u0) [above right of=30]{};
\draw (00) -- (10)-- (20)-- (30);
\draw (30) [dotted] -- (u0);
\node[main node,label={west:$(0 \ltimes 0,1)$}] (01) [above left of=00]{};
\draw (00)--(01);
\node[main node,label={west:$(0 \ltimes 1,1)$}] (11) [above left of=10]{};
\draw (10)--(11);
\node[main node,label={west:$(0 \ltimes 2,1)$}] (21) [above left of=20]{};
\draw (20)--(21);
\node[main node,label={west:$(0 \ltimes 3,1)$}] (31) [above left of=30]{};
\draw (30)--(31);
\node[mnode] (u1) [above right of=31]{};
\draw (01) -- (11)-- (21)-- (31);
\draw (31) [dotted] -- (u1);
\draw (00)--(01);
\node[main node,label={east:$( 1 \ltimes 0,0)$}] (w0a) [above right of=u0]{};
\node[main node,label={west:$( 1 \ltimes 0 ,1)$}] (w1a) [above right of=u1]{};
\draw (w0a)--(w1a);
\node[main node,label={east:$(1 \ltimes 1,0)$}] (10a) [above right of=w0a]{};
\node[main node,label={east:$(1 \ltimes 2,0)$}] (20a) [above right of=10a]{};
\node[mnode] (u0a) [above right of=20a]{};
\draw (w0a) -- (10a)-- (20a);
\draw (20a) [dotted] -- (u0a);
\draw (w0a)--(w1a);
\node[main node,label={west:$(1 \ltimes 1,1)$}] (11a) [above right of=w1a]{};
\draw (10a)--(11a);
\node[main node,label={west:$(1 \ltimes 2,1)$}] (21a) [above left of=20a]{};
\draw (20a)--(21a);
\node[mnode] (u1a) [above right of=21a]{};
\draw (w1a) -- (11a)-- (21a);
\draw (21a) [dotted] -- (u1a);
\node[main node,label={east:$( 2 \ltimes 0,0)$}] (w0b) [above right of=u0a]{};
\node[main node,label={west:$( 2 \ltimes 0 ,1)$}] (w1b) [above right of=u1a]{};
\draw (w0b)--(w1b);
\node[main node,label={east:$(2 \ltimes 1,0)$}] (10b) [above right of=w0b]{};
\node[main node,label={east:$(2 \ltimes 2,0)$}] (20b) [above right of=10b]{};
\node[mnode] (u0b) [above right of=20b]{};
\draw (w0b) -- (10b)-- (20b);
\draw (20b) [dotted] -- (u0b);
\draw (w0b)--(w1b);
\node[main node,label={west:$(2 \ltimes 1,1)$}] (11b) [above right of=w1b]{};
\draw (10b)--(11b);
\node[main node,label={west:$(2 \ltimes 2,1)$}] (21b) [above left of=20b]{};
\draw (20b)--(21b);
\node[mnode] (u1b) [above right of=21b]{};
\draw (w1b) -- (11b)-- (21b);
\draw (21b) [dotted] -- (u1b);
\node[mnode] (u0c) [above right of=u0b]{};
\node[mnode] (u1c) [above right of=u1b]{};
\node[mnode] (u0d) [above right of=u0c]{};
\node[mnode] (u1d) [above right of=u1c]{};
\draw (u0c) [dotted] -- (u0d);
\draw (u1c) [dotted] -- (u1d);
\node[main node,label={east:$(\omega^2 ,0)$}] (w0e) [above right of=u0d]{};
\node[main node,label={north:$(\omega^2 ,1)$}] (w1e) [above right of=u1d]{};
\draw (w0e)--(w1e);
\end{tikzpicture}
\smallskip

The lattice $L$ in Theorem  \ref{lat}.
\end{center}

\begin{proof}
As in  the proof of Proposition \ref{latprop},
$ M = (K \times \{ 1 \}) \cup \{ (0,0) \}  $ is a proper complete sublattice of 
$\mathbff L$  and 
$\langlee M , ( \omega^2 , 0) \ranglee = L$,
hence $( \omega^2 , 0)$ is a relative generator.

We now show  that, for every $n \in \mathbb N$, $n \geq 1$, 
the element  $(n \ltimes 0, 0)$
is a non-generator. 
Suppose that $n \geq 1$ and 
$\langlee X  ,  (n \ltimes 0, 0) \ranglee  = L$.
We first claim that  
$ (n \ltimes 1, 0) \in \langlee X \ranglee $. 
Indeed, since $L \setminus \{ (n \ltimes 1, 0), (n \ltimes 1, 1) \} $ 
is a complete sublattice of $L$, then either
$(n \ltimes 1, 0) \in X $, or $(n \ltimes 1, 1)  \in X$.
If the first eventuality occurs, we are done.
Otherwise, let $P=  \{ \, (k, 0)  \mid 
k \in K \text{ and }  \ k \geq n \ltimes 1 \text{  in } K  \, \} $.
Since $L \setminus P$ is a complete sublattice
of $L$ and $\langlee X  ,  (n \ltimes 0, 0) \ranglee  = L$, then   
$(k, 0) \in X$, for some $k  \geq n \ltimes 1$.  
From $(n \ltimes 1, 1)  \in X$ and
$(k, 0) \in X$ we get
$(n \ltimes 1, 1) \wedge (k, 0) = (n \ltimes 1, 0)
\in \langlee X   \ranglee $.  
The claim that $ (n \ltimes 1, 0) \in \langlee X \ranglee $
has been proved.

Now let $m \in \mathbb N$, $m \geq 1$.
As above, since $L \setminus \{ (n{-}1  \ltimes m, 0),
(n{-}1  \ltimes m, 1) \} $ 
is a complete sublattice of $L$, then either
$(n{-}1  \ltimes m, 0) \in X $, or $(n{-}1  \ltimes m, 1)  \in X$.
Since $ (n \ltimes 1, 0) \in \langlee X \ranglee $
and $ (n{-}1  \ltimes m, 1)  \wedge (n \ltimes 1, 0) = (n{-}1  \ltimes m, 0) $,
we get  $(n{-}1  \ltimes m, 0) \in \langlee X \ranglee$
in each case. Since $m \geq 1$ was arbitrary in the above argument  
 and $\langlee X \ranglee$ is a complete sublattice of $L$, then
$\bigvee _{ m \geq 1 } (n{-}1  \ltimes m, 0) 
= (n  \ltimes 0, 0) \in \langlee X \ranglee$, thus
$\langlee X \ranglee = L $ follows from 
$\langlee X  ,  (n \ltimes 0, 0) \ranglee  = L$.

We have proved that 
the elements of the form
$(n \ltimes 0, 0)$, $n \geq 1$, 
are non-generators. Since
$( \omega^2 , 0)  $
is a relative generator
and  $ ( \omega^2 , 0)  = \bigvee _{ n \in \mathbb N} (n \ltimes 0, 0)$, 
then the set of all the non-generators
is not a complete sublattice of $\mathbff L$. 

The proof is complete. We just point out that,
arguing as above,
 we get 
\begin{equation*}    
\Gamma= \{ \, (n \ltimes 0, 0)  \mid  n \in \mathbb N \, \}
\cup \{ \, (n \ltimes 0, 1)  \mid  n \in \mathbb N, \ n \geq 1 \, \}
\cup \{ ( \omega ^2,1)  \}  
  \end{equation*} 
and $\Phi= \langle \Gamma \rangle = \Gamma \cup  \{ ( \omega ^2,0)  \} $.
Using the arguments in Remark \ref{pf}(c),
the  maximal proper complete sublattices
of $L$ are 
$L \setminus \{ (0 \ltimes 0, 1)$ and
$L \setminus \{ (n \ltimes m, 0),
(n  \ltimes m, 1) \} $, for $n,m \in \mathbb N$ 
and $m \geq 1$.
\end{proof}

The complete lattice $L$ in Theorem \ref{lat}
is countable. The example can be somewhat simplified 
if we allow lattices of larger cardinalities.
 In the following proposition $[0,1]$
is the closed interval of real  numbers between
$0$ and  $1$, with the lattice operations
of $\sup$ and $\inf$.  

\begin{proposition} \labbel{01}
In the complete lattice $L=[0,1] \times \{  0,1\} $
all the elements are non-generators, except for
$(0,1)$ and $(1,0)$.
Hence the set $\Gamma$   of the non-generators
is not a complete sublattice. The set $\Gamma$ 
generates the whole of $[0,1] \times \{  0,1\} $.
 \end{proposition}  
 
\begin{proof} (Sketch)
Let $0 < r < 1$
and suppose that $\langlee X, (r,0) \ranglee = L$.
Since $L \setminus \{ \, (s,0) \mid r < s \, \} $
is a complete sublattice, there is $s>r$
such that $(s,0) \in X$.      
For every $t$ with
$0 \leq t  <r$, the set
 $ L \setminus \big(]t,r[ \times \{ 0,1 \} \big ) $
is a complete sublattice,  hence
$(v,a) \in X$, for some $a \in \{ 0, 1 \} $
and $v$ with $t < v <r$. Taking the meet
with     $(s,0)$, we get  
$(v,0) \in \langlee X \ranglee $.
Letting $t$ approximate 
$r$ from below, we get a sequence 
of elements  $(v_n,0) \in \langlee X \ranglee $
whose join is $(r,0)$, hence
 $(r,0) \in \langlee X \ranglee $,
thus $ \langlee X \ranglee=L $.
 
All the rest is symmetrical or 
similar to the proof of Theorem \ref{lat}. 
\end{proof}
 
As we have just showed, the set $\Gamma$ of the non-generators
of a complete lattice is not necessarily a complete sublattice. 
However, $\Gamma$ is always a sublattice, that is, $\Gamma$
is closed under finite meets and joins. Indeed, if $a$ and  $b$
are non-generators and, say,   $\langlee X, a \wedge b \ranglee = L $,
then  $\langlee X, a , b \ranglee = L $, since 
 $\langlee X, a \wedge b \ranglee  \subseteq \langlee X, a ,  b \ranglee  $.
Hence  $\langlee X, a \ranglee = L $, since $b$ is a non-generator, thus 
 $\langlee X \ranglee = L $, since $a$ is a non-generator. This shows 
that $a \wedge b$ is a non-generator. 

\begin{remark} \labbel{card}
(a) As we mentioned, we have dealt with complete lattices and semilattices
just for the sake of simplicity.
Since the counterexample in Theorem \ref{lat}
is countable, completeness is the same as countable completeness.
Hence even in a countably complete lattice it might happen that 
 the set of all the non-generators fails to be  a substructure. 
Recall that \emph{countably complete}
means that every countable subset has a meet and a join. 

(b)
Similarly, the version of Proposition \ref{seml}
 holds with the same proof when ``complete'' is everywhere replaced 
by ``countably complete'' or, more generally,
by ``${<}\kappa$-complete'', for $\kappa$ an infinite cardinal,
where the latter notion means that every subset of cardinality
$ < \kappa $ has a meet.   
 Of course, the notion of meet-reducibility in the proof
should be replaced by an appropriate notion
of ${<}\kappa$-meet reducibility.

The versions of Proposition \ref{seml}
 hold also under the convention under which
``completeness'' does not include the possibility of
taking the meet of the empty set, that is, for semilattices
not assumed to have a maximum.   

(c) Notice that the notion of a non-generator in a 
${<}\kappa$-complete semilattice
depends on $\kappa$.
For example, consider a descending countably infinite sequence
with a further element $d$ added at the bottom. 
The element $d$ is indispensable 
if the above  semilattice is considered as finitary; however, $d$
is a non-generator if we consider the semilattice as 
countably complete. 

Similarly, if $d$ is added at the bottom
of a descending chain of cofinality $\lambda$,
then $d$
is a non-generator if and only if  we consider the semilattice as 
${<}\kappa$-complete, for $\kappa> \lambda $.

(d) In the proof of Theorem \ref{lat} we have only used binary meets
and (countable) infinitary joins, hence the theorem holds 
in the context of countably-join-complete lattices.
\end{remark}

\begin{corollary} \labbel{unasola}
There is an algebraic structure
with a single operation depending on countably many arguments
and such that the set of all the non-generators
fails to be a substructure.
 \end{corollary} 

\begin{proof}
Consider the example from 
Theorem \ref{lat} with the single infinitary operation
 $f(x_0, x_1, x_2, \dots)  =  \bigvee _{ i \in \mathbb N}  ( x_{2i} \wedge x_{2i+1})  $,
noticing that  finite meets and countable joins
can be expressed in function of $f$. Then 
use remark  (d) above.
 \end{proof}

Acknowledgement. 
We thank an anonymous referee for many useful comments which helped 
improve the paper.

\end{document}